\documentclass[a4paper,11pt,reqno]{amsart}
\usepackage{enumerate}
\usepackage{graphicx}
\usepackage{amsfonts}
\usepackage{graphicx}
\usepackage{amsmath}
\usepackage{amsxtra}
\usepackage{amssymb}
\usepackage[latin1]{inputenc}
\usepackage{dsfont} 

\newtheorem{theorem}{Theorem}[section]

\newtheorem{corollary}[theorem]{Corollary}

\newtheorem{definition}[theorem]{Definition}

\newtheorem{lemma}[theorem]{Lemma}

\newtheorem{proposition}[theorem]{Proposition}
\newtheorem{remark}[theorem]{Remark}

\numberwithin{equation}{section}

\newcommand{\cG}{\mathcal{G}}

\newcommand{\Hag}{H_{\alpha\sigma_M}}
\providecommand{\C}[1]{\mathcal{#1}}

\newcommand{\G}{\Gamma}

\newcommand{\Sphere}{{\mathbb S}}

\newcommand{\RR}{{\mathbb R}}
\newcommand{\R}{{\mathbb R}}

\newcommand{\CC}{{\mathbb C}}
\newcommand{\NN}{{\mathbb N}}

\newcommand{\KK}{{\mathbb K}}
\newcommand{\supp}{\mathrm{supp}\,}
\providecommand{\C}[1]{\mathcal{#1}}
\newcommand{\aaa}{\C{E}}
\newcommand{\cE}{\C{E}}
\newcommand{\DD}{\C{D}}
\newcommand{\cD}{\C{D}}
\newcommand{\capp}{\mathrm{cap}}

\newcommand{\phim}{\Phi^{-1}}

\begin{document}
%
\title[Eigenfunctions and spectral theory for forms]{Generalized eigenfunctions and spectral theory for strongly local Dirichlet forms}
\author[D.~Lenz]{Daniel Lenz$^1$}
\address{$^1$ Mathematisches Institut, Friedrich-Schiller Universit\"at Jena,
  Ernst-Abb\'{e} Platz 2, 07743 Jena, Germany}
\email{ daniel.lenz@uni-jena.de }
\urladdr{http://www.analysis-lenz.uni-jena.de/}

\author[P.~Stollmann]{Peter Stollmann$^2$}
\address{$^2$ Fakult\"at f\"ur Mathematik,
           Technische Universit\"at, 09107 Chemnitz, Germany }
\email{ peter.stollmann@mathematik.tu-chemnitz.de }
\author[I.~Veseli\'c]{Ivan Veseli\'c$^3$}
\address{$^3$Emmy-Noether-Project \emph{Schr\"odingeroperatoren}, Fakult\"at f\"ur Mathematik,\, 09107\, TU Chemnitz, Germany   }
\urladdr{http://www.tu-chemnitz.de/mathematik/enp/}



\date{June 30, 2009}

\begin{abstract}
We present an introduction to the framework of strongly local Dirichlet forms and discuss connections between the existence of certain generalized eigenfunctions and spectral properties within this framework. The range of applications is illustrated by a list of examples.

\end{abstract}
\maketitle
\section*{Introduction}
There is a long history to the study of  connections between the spectrum of a selfadjoint differential  operator and properties of generalized solutions to the associated  eigenvalue equation.
In this context, the  following two 'meta theorems' have attracted particular attention:

\begin{itemize}

\item Positive generalized eigenfunctions exist for energies below the spectrum and   the spectrum begins at the energy, where positive generalized eigenfunctions cease to exist.

\item The spectrum is  given by those energies, for which a (suitably) bounded generalized solution exists.
\end{itemize}

 The first statement is sometimes discussed under the name of 'Allegretto Piepenbrink theorem'. The second statement is discussed under the heading  of 'Shnol theorem'. Precise versions (and proofs) of these statements have been given in various contexts. It turns out that the  framework of (strongly local) Dirichlet forms allows one to give a unified and structurally rather simple discussion of these two results. This has recently be   shown in \cite{LenzSV} (for the Allegretto Piepenbrink theorem) and in \cite{BoutetdeMonvelLS-08} for the Shnol type result,
see the results on expansion in eigenfunctions in \cite{BoutetdeMonvelS-03b} as well.
The mentioned framework includes a variety of operators among them  Schr\"odinger operators, (uniform) elliptic operators on manifolds and (suitable) quantum graphs. Accordingly, the mentioned results have a rather broad  applicability.

Our aim here is to discuss this approach to basic spectral theory via Dirichlet forms in a way that is accessible to the non-specialist. In this way, we  will not only feature the  Shnol Theorem and the Allegretto Piepenbrink theorems of \cite{BoutetdeMonvelLS-08,LenzSV}  but also hope  to advertise the use of Dirichlet forms in spectral problems.  For this reason we  also conclude the paper with  a discussion of  various applications.
The results in this paper are concerned with strongly local Dirichlet forms. A study of similar results for non-local Dirichlet forms (e.g. graphs) can be found  in \cite{FrankLWpre}.

\medskip

The organisation of this paper is as follows. We give a introduction into Dirichlet froms in Sections \ref{Strongly}, \ref{Measure} and \ref{Weak}.
This introduction is aimed at a non-specialist. We then discuss a version of Allegretto-Piepenbrink Theorem in Section \ref{Positive} and results related to Shnol's Theorem in Section \ref{Weaksolutions}. These sections contain sketches of ideas and proofs. Finally, we discuss applications  in Section \ref{Examples}.

\section{Strongly local Dirichlet forms}\label{Strongly}
In this section we describe the set-up used throughout the paper.
We refer to \cite{Fukushima-80} as to the
classical standard reference as well as
\cite{BouleauH-91,Davies-90b,FukushimaOT-94,MaR-92} for literature on
Dirichlet forms. We treat real and complex function
spaces at the same time and write $\KK$ to denote either $\RR$ or $\CC$.

Throughout we will work with a locally compact, separable metric space
$X$ endowed with a positive Radon measure $m$ with $ \supp m=X$.

\subsection*{Dirichlet forms}
The central object of our studies is a regular
Dirichlet form $\mathcal{E}$ with domain $\mathcal{D}$ in $L^2(X)$ and
the selfadjoint operator $H_0$ associated with $\mathcal{E}$. In order to precisely define these notions we
 recall the basic terminology of Dirichlet forms:
Consider a dense subspace $\mathcal{D} \subset L^2(X,m)$  and a
sesquilinear and non-negative map $\mathcal{E}\colon\mathcal{D} \times \mathcal{D} \rightarrow \KK$
such that $\mathcal{D}$ is closed with respect to the energy
norm $\|\,\cdot\,\|_\mathcal{E}$, given by
$$
\|u\|_\mathcal{E}^2=\mathcal{E}[u,u] +\| u\|_{L^2(X,m)}^2,
$$
in which case one speaks of a \textit{closed form} in $L^2(X,m)$. In the sequel we will write
$$\mathcal{E}[u]:= \mathcal{E} [u,u]. $$
  The selfadjoint operator $H_0$
associated with $\mathcal{E}$ is then characterized by
$$
D(H_0)\subset \mathcal{D} \ \mbox{and } \mathcal{E}[f,v]=(H_0f\mid
v)\quad (f\in D(H_0), v\in \mathcal{D}).
$$
Such a closed form is said to be a \textit{Dirichlet form} if
$\mathcal{D}$ is stable under certain pointwise operations; more
precisely, $T:\KK\to\KK$ is called a \emph{normal contraction} if
$T(0)=0$ and $|T(\xi)-T(\zeta)|\le |\xi -\zeta|$ for any
$\xi,\zeta\in\KK$ and we require that for any $u\in \mathcal{D}$ also
$$
T\circ u\in \mathcal{D}\mbox{ and }\mathcal{E}[T\circ u]\le
\mathcal{E}[u].
$$
In the real case, this condition is often  replaced by equivalent
but formally weaker statement involving $u\vee 0$ and $u\wedge 1$, see
\cite{Fukushima-80}, Thm. 1.4.1 and  \cite{MaR-92}, Section I.4.

A Dirichlet form is called \textit{regular} if $\mathcal{D} \cap
C_c(X)$ is large enough so that it is dense both in $(\mathcal{D},\| \cdot \|_\mathcal{E})$ and
$(C_c(X),\| \cdot \|_{\infty })$, where $C_c(X)$ denotes the space of
continuous functions with compact support.

\subsection*{Examples}
Here, we discuss some examples showing the wide range of applicability of Dirichlet forms.

\medskip

\textit{The Laplacian on Euclidean space.} The Laplacian in Euclidean space is the typical example to be kept in mind. It is given by
$$H_0=-\Delta \mbox{ on }L^2(\Omega ),\quad \Omega \subset \mathbb{R}^d\mbox{ open, }$$
in which case
$$
\mathcal{D}=W^{1,2}_0(\Omega )\mbox{ and
}\mathcal{E}[u,v]=\int_{\Omega }(\nabla u| \nabla v)dx.$$

Note that for  differentiable contractions $T : \RR\longrightarrow \RR$ the chain rule easily gives the crucial Dirichlet form property for real valued functions $u$  as
$$ \mathcal{E}(Tu) = \int_\Omega  ( \nabla Tu, \nabla Tu) dx = \int_\Omega |T'(u(x))|^2 (\nabla u, \nabla u) dx \leq \mathcal{E} (u)$$
as $|T'(z)| \leq 1$ for all $z\in\CC$.

\medskip

\textit{Uniform elliptic operators in Euclidean space:}
In the previous example we can allow for quite irregular coefficients of the differential operator.
More precisely, let $\Omega\subset \R^d$ open and  let $A$ be a measurable map from $\Omega$ into the symmetric $d\times d$ matrices. Assume that there exist $c,C>$ such that the eigenvalues of $A(x)$ lie in $[c,C]$ for all $x\in \Omega$. Then, the form $\mathcal{E}_A$ defined on $W^{1,2}_0(\Omega )$ by
$$ \mathcal{E}_A[u,v]=\int_{\Omega }(A(x) \nabla u| \nabla v)dx$$
is a regular Dirichlet form.

\medskip

\textit{Laplace Beltrami  and unifom elliptic operators on manifolds:} The previous example can easily be generalized to {Laplace Beltrami} operators on Riemannian manifolds:  Let $M$ be a Riemannian manifold with metric tensor $g$ and exterior derivative $d$. Then, the form
$$\mathcal{E}_c (u,v) :=\int_M (d u, d v) dx$$
defined for  $u,v\in C_c^\infty (M)$ is closable. The closure is a Dirichlet form and its domain of definition is given by  $W^{1,2}_0 (M)$. The generator is the Laplace Beltrami operator.
Again, we can allow for a measurable map $A$ from $M$ into the symmetric linear maps on the corresponding cotangent spaces with eigenvalues lying in some interval $[c,C]$ for $c,C>0$ and obtain the Dirichlet form
$$ \mathcal{E}_A (u,v) :=\int_M (A \, d u, d v) dx$$
defined on $W^{1,2}_0 (M)$. These examples can be further generalized to allow for some subriemannian manifolds. We will not give details here.

\medskip

\textit{Quantum graphs with Kirchhoff boundary conditions}: This example has received attention in recent times. We refrain from giving details here but refer to the last  section of the paper.

\medskip

\subsection*{Capacity}
The \emph{capacity} is a set function that allows one  to measure the size of sets in a way that is adapted to the form $\aaa$.

\medskip

For $U\subset X$, $U$ open, we define
$$
\capp(U):=\inf\{\| v \|^2_\mathcal{E} \mid v\in\DD, \chi_U\le v\},$$
where we set $ (\inf\emptyset =\infty )$.
For arbitrary $A\subset X$, we then set
$$
\capp(A):=\inf\{ \capp(U) \mid A\subset U\}
$$
(see \cite{Fukushima-80}, p. 61f.).  We say that a property holds
\emph{quasi-everywhere}, short \emph{q.e.}, if it holds outside a set of
capacity $0$. A function $f:X\to\KK$ is said to be \emph{quasi-continuous},
\emph{q.c.} for short, if, for any $\varepsilon >0$ there is an open set
$U\subset X$ with $\capp(U)\le \varepsilon$ so that the restriction of $f$ to
$X\setminus U$ is continuous.

A fundamental result in the theory of Dirichlet forms says that every
$u\in\DD$ admits a q.c. representative $\tilde{u}\in u$ (recall that $u\in
L^2(X,m)$ is an equivalence class of functions) and that two such q.c.
representatives agree q.e. Moreover, for every Cauchy sequence $(u_n)$ in
$(\mathcal{D},\| \cdot \|_\mathcal{E})$ there is a subsequence $(u_{n_k})$
such that the $(\tilde{u}_{n_k})$ converge q.e. (see \cite{Fukushima-80},
p.64f).

Whenever we will write expressions containing pointwise evaluations of functions  $u$ in the future, we will  assume that a quasi continuous representative has been chosen.

\subsection*{Strong locality and the energy measure}

\medskip

$\mathcal{E}$ is called \textit{strongly local} if
$$\mathcal{E}[u,v]=0$$
whenever $u$ is constant a.s. on the support of $v$.

Every strongly local, regular Dirichlet form $\mathcal{E}$
can be represented in the form
\[
\mathcal{E}[u,v] = \int_X d\Gamma (u,v)
\]
where $\Gamma $ is a nonnegative sesquilinear mapping from
$\mathcal{D}\times\mathcal{D}$ to the set of $\KK$-valued Radon
measures on $X$. It is determined by
\[
\int_X \phi\, d\Gamma(u,u) = \mathcal{E}[u,\phi u] -\frac12
\mathcal{E}[u^2,\phi ]
\]
for realvalued $u\in\DD$, $\phi\in \mathcal{D} \cap
C_c(X)$ and called \textit{energy measure}; see also \cite{BouleauH-91}.

Obviously, all examples discussed in the preceeding subsection are strongly local. In the case of the Laplacian in Euclidean space, the measure $\Gamma$ is given by $(\nabla u |\nabla v)dx$ appearing above.

We discuss properties of the energy measure next (see e.g. \cite{BouleauH-91, Fukushima-80, Sturm-94b}).

The energy measure inherits strong locality from $\aaa$ viz $\chi_U d\Gamma (\eta,u)=0$ holds
 for any open $U\in X$ and any $\eta,u\in \DD$ with $\eta$ constant on $U$.  This directly  allows one to extend $\Gamma$  to $\mathcal{D}_{\text{loc}} $  defined as
\[
\lbrace u\in L^2_{\text{loc}}\mid \text{for all compact }
K \subset X \text{ there is  }
\phi \in \mathcal{D} \text{ s.\,t. }
\phi =u  \ m\text{-a.e.{} on }K\},
\]
We will denote this extension by $\Gamma$ again.  This extension is \textbf{strongly local} again i.e. satisfies
\label{local}
\begin{equation*}
  \chi_U d\Gamma (\eta,u) = 0,
\end{equation*}
for any open $U\in X$ and any $\eta,u\in \DD_{\text{loc}}$ with $\eta$ constant on $U$.
The set $\DD$ is then given as the set of all $u\in \DD_{\text{loc}}$ with $\int 1 d\Gamma (u) < \infty$.
The energy measure satisfies the \textbf{Leibniz rule},
\[
d\Gamma(u\cdot v,w)=ud\Gamma (v,w)+vd\Gamma (u,w),
\]
for all $u,v\in \mathcal{D}_{\text{loc}}\cap L^\infty_{\text{loc}} (X)$.
(In fact strong locality is of $\aaa$ is   equivalent to the
validity of the Leibniz rule for functions in $\DD\cap L^\infty_{\text{loc}}$.)
The energy measure also satisfies
the chain rule
\[
d\Gamma (\eta (u),w)=\eta'(u)d\Gamma (u,w)
\]
whenever  $u,w\in \mathcal{D}_{\text{loc}}\cap L^\infty_{\text{loc}}$ are real valued and $\eta$ is continuously differentiable.

We write $d\Gamma(u):=d\Gamma(u,u)$ and note that the energy measure
satisfies the \textbf{Cauchy-Schwarz inequality}:
\begin{eqnarray*}
  \int_X|fg|d|\Gamma(u,v)| & \le &
  \left(\int_X|f|^2d\Gamma(u)\right)^{\frac12}\left(\int_X|g|^2d\Gamma(v)\right)^{\frac12}\\
  & \le & \frac12 \int_X|f|^2d\Gamma(u)+ \frac12\int_X|g|^2d\Gamma(v)
\end{eqnarray*}
for all $u,v\in \mathcal{D}_{\text{loc}}$ and $f,g :X\longrightarrow \CC$ measurable.

Due to Leibniz rule the sets $\DD$ and $\DD_{\text{loc}}$ resp. have certain closedness properties under multiplication. This is an interesting feature and we discuss it next.

It is not hard to see that any function in $u\in \mathcal{D}_{\text{loc}}$ with compact support belongs in fact to $\DD$ (as $\int d\Gamma (u) = \int_{\supp u} d\Gamma (u) < \infty$). More generally,  localized versions of functions from $\mathcal{D}_{\text{loc}}$ belong to $\DD$. More precisely, the following holds \cite{LenzSV}.

\begin{lemma}\label{localversion}
\begin{itemize}
 \item [(a) ] Let $\Psi\in \mathcal{D}_{\text{loc}}\cap L^\infty_ {\text{loc}}(X)$ and $\varphi \in \mathcal{D}\cap L^\infty_c (X)$ be given. Then, $\varphi \Psi$ belongs to $\mathcal{D}$.
 \item [(b) ] Let $\Psi\in \mathcal{D}_{\text{loc}}$ and $\varphi \in \mathcal{D}\cap L^\infty_c (X)$ be such that $d\Gamma(\varphi)\le C\cdot dm$. Then, $\varphi \Psi$ belongs to $\mathcal{D}$.
\end{itemize}
\end{lemma}

Part (a) of the lemma gives in particular that  $\DD\cap C_c (X) = \DD_{\text{loc}}\cap C_c (X)$ and $\DD \cap L^\infty_c (X) = \DD_{\text{loc}}\cap L^\infty_c (X)$  are closed under multiplication.

\medskip

In order to introduce weak solutions on open subsets $U$  of $X$, we extend $\aaa$
 to $\mathcal{D}_{\text{loc}}(U)\times \mathcal{D}_c (U)
$: where,
$$
\mathcal{D}_{\text{loc}}(U):=
\lbrace u\in
L^2_{\text{loc}}(U)\mid \forall \text{compact}\;\:
K \subset U \exists\;
\phi \in \mathcal{D} \text{ s.\,t. }
\phi =u  \ m\text{-a.e.{} on }K\}
$$
$$
\mathcal{D}_c(U):=
\lbrace \varphi\in \mathcal{D}|\supp\varphi\mbox{  compact in }U\rbrace .
$$
For $u\in \mathcal{D}_{\text{loc}}(U),
\varphi\in \mathcal{D}_c(U)$ we define
$$
\aaa[u,\varphi]:=\aaa[\eta u,\varphi].
$$
Here,  $\eta \in \mathcal{D}\cap C_c(U)$ is arbitrary with constant value
$1$ on the support of $\varphi$. This makes sense as the RHS does not depend
on the particular choice of $\eta$ by strong locality.

Obviously,  also $\Gamma$ extends
to a mapping $\Gamma:\mathcal{D}_{\text{loc}}(U)\times
\mathcal{D}_{\text{loc}}(U)\to \C{M}_R(U)$.

\subsection*{The intrinsic metric, strict locality and cut-off functions}
Using the energy measure one can define the \textit{intrinsic metric}
$$\rho: X\times X\longrightarrow [0,\infty]$$ by
$$
\rho (x,y)=\sup \lbrace |u(x)-u(y)|\ | u\in
\mathcal{D}_{\text{loc}}\cap C(X) \mbox{ and } d\Gamma (u)\leq
dm\rbrace
$$
where the latter condition signifies that $\Gamma (u)$ is absolutely
continuous with respect to $m$ and the Radon-Nikodym derivative is
bounded by $1$ on $X$. Despite its name, in general, $\rho$ need not be a
metric. However, it is a pseudo metric viz it is symmetric, satisfies $\rho (x,x)=0$ for all $x\in X$ and satisfies the triangle inequality.

We say that $\aaa$ is \emph{strictly local} if $\rho$ is a metric that
induces the original topology on $X$.

Note that strict locality  implies that  $X$ is connected, since
otherwise points in $x,y$ in different connected components would give $\rho(x,y)=\infty$, as characteristic functions of connected components are continuous
and have vanishing energy measure.

We denote the intrinsic balls by
$$B(x,r):=\{ y\in X| \rho(x,y)\le r\} .$$
An important consequence of strict locality  is that the distance
function $\rho_x(\cdot):=\rho(x,\cdot)$ itself is a function in
$\mathcal{D}_{\text{loc}}$ with $d\Gamma(\rho_x)\le dm$, see
\cite{Sturm-94b}. This easily extends to the fact that for every
closed $E\subset X$ the function $\rho_E(x):= \inf\{ \rho(x,y)|y\in
E\}$ enjoys the same properties (see the Appendix of \cite{BoutetdeMonvelLS-08}). This has a very
important consequence. Whenever $\zeta : \RR\longrightarrow \RR$ is
continuously differentiable, and $\eta :=\zeta \circ \rho_E$, then
$\eta$ belongs to $\mathcal{D}_{\text{loc}}$ and satisfies
\begin{equation}\label{ac}
  d \Gamma (\eta) = (\zeta'\circ \rho_E)^2 d \Gamma (\rho_E)\leq (\zeta'\circ \rho_E)^2  dm.
\end{equation}
For this reason a lot of good cut-off functions are around in our context.
More explicitly we note the following lemma (Lemma 1.3 in \cite{LenzSV}, see \cite{BoutetdeMonvelLS-08} as well).

\begin{lemma} \label{cutoff} For any compact $K$ in $X$ there exists a $\varphi \in C_c (X)\cap \mathcal{D}$ with $\varphi \equiv 1$ on $K$, $\varphi \geq 0$ and $d\Gamma (\varphi)\leq C \, dm$ for some $C>0$. If $L$ is another compact set containing $K$ in its interior, then $\varphi$ can be chosen to have support in $L$.
\end{lemma}

\subsection*{Irreducibility}
We will now discuss a notion that will be crucial in the proof of the
existence of positive weak solutions below the spectrum. In what follows,
$\mathfrak{h}$ will denote a densely defined, closed semibounded form in
$L^2(X)$ with domain $D(\mathfrak{h})$ and positivity preserving semigroup
$(T_t;t\ge 0)$.  We denote by $H$ the associated operator.  Actually, the
cases of interest in this paper are the situation that $\mathfrak{h}=\aaa$ is a
Dirichlet form as discussed above, or a measure perturbation thereof
$\mathfrak{h}=\aaa+\nu$. Here it is assumed that the positive and negative part of the measure $\nu$ obey
$\nu_+\in\ \mathcal{M}_{R,0}, \nu_-\in \mathcal{M}_{R,1}$,
where the classes
$\mathcal{M}_{R,0}, \mathcal{M}_{R,1}$ are discussed in the next section.

We say that $\mathfrak{h}$ is \emph{reducible}, if there is a
measurable set $M\subset X$ such that $M$ and its complement $M^c$ are
nontrivial (have positive measure) and $L^2(M)$ is a reducing subspace for
$M$, i.e., ${\mathds{1}}_MD(\mathfrak{h})\subset D(\mathfrak{h})$,
$\mathfrak{h}$ restricted to ${\mathds{1}}_MD(\mathfrak{h})$ is a closed form
and $\aaa(u,v) = \aaa(u{\mathds{1}}_M,v{\mathds{1}}_M) +\aaa(u{\mathds{1}}_{M^c},v{\mathds{1}}_{M^c}) $
for all $u,v$.
If there is no such decomposition of $\mathfrak{h}$, the latter form is called
\emph{irreducible}. Note that reducibility can be rephrased in terms of the
semigroup and the resolvent:

\begin{theorem}
 Let $\mathfrak{h}$ be as above. Then the following conditions are equivalent:
\begin{itemize}
 \item[(i)] $\mathfrak{h}$ is irreducible.
  \item [(ii)]$T_t$ is positivity improving, for every $t>0$, i.e. $f\ge 0$ and $f\not=0$ implies that $T_tf>0$ a.e.
 \item[(iii)] $(H+E)^{-1}$ is positivity improving for every $E<\inf\sigma(H)$.
\end{itemize}
\end{theorem}
We refer to \cite{ReedS-78}, XIII.12 and a forthcoming paper \cite{LenzSV-pre}
for details.

It is quite easy to see that a disconnected space easily leads to reducible forms. The converse is not exactly right, but there are recent results that go far in this direction and characterize irreducibility, cf \cite{terEstR-08}.

\subsection*{Assumptions}
For the convenience of the reader we gather in this section assumptions and notation used in the sequel.

\bigskip

We will exclusively deal with regular, strongly local Dirichlet forms $\aaa$. The corresponding energy measure is denoted by $\Gamma$. The associated intrinsic metric is denoted by $\rho$.

\smallskip

We always chose quasi continuous representatives for elements of $\DD_{\text{loc}}$.

\smallskip

The strongly local Dirichlet form $\aaa$  is \emph{strictly local} if $\rho$ is a metric that
induces the original topology on $X$. This condition can be slightly weakened.  It suffices to assume that the pseudometric $\rho$ induces the original topology on $X$ (as for a given $x\in X$ one can then always restrict attention to the set of $y$ with $\rho(x,y)<\infty$).

Later we will also encounter a growth assumption on the intrinsic metric.
\begin{itemize}
\item[(G)] \label{A1} All intrinsic balls have finite volume with subexponential
  growth:
$$
e^{-\alpha\cdot R}m(B(x,R))\to 0\mbox{ as }R\to\infty\mbox{ for all
}x\in X, \alpha>0 .$$
\end{itemize}

Finally, we note that $\aaa$ is called \textit{ultracontractive} if  for each $t>0$ the  semigroup  $e^{-tH_0}$ gives a map from  $L^2(X)$ to $L^\infty (X)$.

\section{Measure perturbations}\label{Measure}
We will be dealing with Schr\"odinger type operators, i.e.,
perturbations $H=H_0+V$, where $H_0$ is associated to a strictly local Dirichlet form and   the function $V$ is a suitable potential. In fact, we can
even include measures as potentials. Here, we follow the approach from   \cite{Stollmann-92,StollmannV-96}. Measure perturbations have been regarded
by a number of authors in different contexts, see e.g.
\cite{AlbeverioM-91,Hansen-99a,Sturm-94b} and the references there.

\medskip

We denote by
$\mathcal{M}_{R}(U)$ the signed Radon measures on the open subset $U$ of $X$
and by $\mathcal{M}_{R,0}(U)$ the subset of measures
$\nu$ that do not charge sets of capacity
$0$, i.e., those measures with $\nu(B)=0$ for every Borel set $B$ with
$\capp(B)=0$. In case that $\nu=\nu_+-\nu_-\in \mathcal{M}_{R,0}(X)$ we can define
$$
\nu[u,v]=\int_X\tilde{u}\overline{\tilde{v}}d\nu\mbox{  for  }u,v\in\DD
\mbox{  with  }\tilde{u},\tilde{v}\in L^2(X,\nu_++\nu_-) .
$$
Of course, a special instance of such  measures is given by $\nu=Vdm$ whenver $V$ belongs to  $ L^1_{\text{loc}}(X)$.

We
 have to rely
upon more restrictive assumptions concerning the negative part $\nu_-$
of our measure perturbation. We write $\mathcal{M}_{R,1}$ for those
measures $\nu\in \mathcal{M}_{R}(X)$ that are $\aaa$-bounded with bound less than one; i.e.
measures $\nu$ for which there is a $\kappa<1$ and a $c_\kappa\geq 0$ such that
$$
\nu[u,u]\le \kappa\aaa[u] + c_\kappa\| u\|^2 .
$$
The set $\mathcal{M}_{R,1}$ can easily be seen to be a subset of
$\mathcal{M}_{R,0}$.

By the KLMN theorem (see \cite{ReedS-75}, p. 167), the sum
$\aaa+\nu$ given  by $D(\aaa+\nu)=\{u\in\DD\mid \tilde{u}\in L^2(X,\nu_+)\}$
is closed and densely defined (in fact $\mathcal{D} \cap
C_c(X)\subset D(\aaa+\nu)$) for
$\nu$ with $\nu_+ \in \mathcal{M}_{R,0}, \nu_-\in \mathcal{M}_{R,1}$ .
We denote the associated selfadjoint operator by $H_0+\nu$.
Note that $\cD \cap L_c^\infty(X)\subset D(\cE+ \nu)$.

An important subclass of $\mathcal{M}_{R,0}$
with very nice properties of the associated operators is the
\textit{Kato class} and the \textit{extended Kato class}. In the present framework it can be defined in the
following way: For $\mu\in \mathcal{M}_0$ and $\alpha >0$ we set
$$
\Phi(\mu,\alpha): C_c(X)_+\to [0,\infty],$$
$$
\Phi(\mu,\alpha)\varphi := \int_X\left(
  (H_0+\alpha)^{-1}\varphi\right){\tilde{ }}\quad d\mu .
$$
The extended Kato class is defined as
$$
\hat{\mathcal{S}}_K:= \{\mu\in \mathcal{M}_0|\exists \alpha>0:
\Phi(\mu,\alpha)\in L^1(X,m)'\}
$$
and, for $\mu\in \hat{\mathcal{S}}_K$ and $\alpha >0$,
$$
c_\alpha(\mu):=\| \Phi(\mu,\alpha)\|_{L^\infty(X,m)}(= \|
\Phi(\mu,\alpha)\|_{L^1(X,m)'}), c_{Kato}(\mu):=\inf_{\alpha>0} c_{\alpha}(\mu)
.
$$
The Kato class is originally defined via the fundamental solution of the
Laplace equation in the classical case. In our setting it consists of those measures $\mu$ with $c_{Kato}(\mu)=0$.

As done in various papers, one can even allow for more singular measures, a direction we are not going to explore here.

As already discussed our measure perturbations preserve closability of the form. They preserve further properties.  In fact, regularity is preserved in our context as well.

\begin{theorem}[\cite{LenzSV}] Let $ (\aaa,\DD)$ be a strongly local, regular Dirichlet form. Let $\nu$ with $\nu_+ \in \mathcal{M}_{R,0}, \nu_-\in \mathcal{M}_{R,1}$ be given.
Then, the perturbed form
$(\cE+\nu,D(\cE+\nu))$ is regular as well.
\end{theorem}

Measure perturbations also preserve irreducibility, as can be seen from the following result.

\begin{theorem}[\cite{LenzSV-pre}] Let  $(\aaa,\DD)$ be a strictly local, regular, irreducible  Dirichlet form. Let $\nu$ with $\nu_+ \in \mathcal{M}_{R,0}, \nu_-\in \mathcal{M}_{R,1}$ be given. Then, the perturbed form
$(\cE+\nu,D(\cE+\nu))$ is irreducible  as well.
 \end{theorem}

\section{Weak solutions} \label{Weak}
Our main aim is to relate properties of weak solutions or generalized eigenfunctions to spectral properties of $H_0 + \mu$. The necessary notation concerning weak solutions is introduced in this section.
Throughout this section we consider a strongly local, regular Dirichlet form,
$(\aaa,\DD)$ on $X$ and denote by $\Gamma: \DD_{\text{loc}}\times\DD_{\text{loc}}\to
\C{M}(X)$ the associated energy measure. We will be concerned with weak
solutions $\Phi$ of the equation
\begin{equation}\label{eq21}
(H_0+V)\Phi=\lambda\cdot \Phi,
\end{equation}
where $H_0$ is the operator associated with $\aaa$ and $V$ is a realvalued,
locally integrable potential. In fact, we will consider a somewhat more
general framework, allowing for measures instead of functions, as presented in
the previous section. Moreover, we stress the fact that (\ref{eq21}) is formal
in the sense that $\Phi$ is not assumed to be in the operator domain of
neither $H_0$ nor $V$. Here are the details.

\medskip

Recall that we could extend $\Gamma$ to a measure valued function on $U$.
In the same way, we can
extend $\nu[\cdot,\cdot]$, using that every $u\in \mathcal{D}_{\text{loc}}(U)$
admits a quasi continuous version $\tilde{u}$.

\begin{definition}
Let $U\subset X$ be open and $\nu\in\C{M}_{R,0}(U)$ be a signed Radon measure on $U$ that charges no set of capacity zero.
Let  $\lambda\in\RR$ and $\Phi\in L^2_{\text{loc}}(U)$. We say that $\Phi$ is a
\emph{weak solution} of $(H_0+\nu)\Phi=\lambda\cdot\Phi$ in $U$ if:
\begin{itemize}
\item[(i)] $\Phi\in\DD_{\text{loc}}(U)$,
\item[(ii)] $\tilde{\Phi}d\nu\in\C{M}_R(U)$,
\item[(iii)] $\forall \varphi\in\DD\cap C_c(U),$
$$
\aaa[\Phi,\varphi]+\int_U\varphi \tilde{\Phi}d\nu = \lambda \cdot (\Phi |\varphi) .
$$
\end{itemize}
If $V\in L^1_{\text{loc}}(U)$ we say that $\Phi$ is a
\emph{weak solution} of $(H_0+V)\Phi=\lambda\cdot\Phi$ in $U$ if it is a weak
solution of  $(H_0+\nu)\Phi=\lambda\cdot\Phi$ for $\nu=Vdm$.
\end{definition}

Next, we briefly discuss these assumptions.
\begin{remark}
\begin{itemize}
\item[(1)] If $\nu=Vdm$ and $V\in L^2_{\text{loc}}(U)$, then property (ii) of the Definition above is satisfied.
\item[(2)] If $\Phi\in L^\infty_{\text{loc}}(U)$ and $\nu\in\C{M}_R(U)$ then (ii) of the Definition above is satisfied.
\item[(3)] If $\nu\in\C{M}_R(U)$ satisfies (ii) above then $\nu - Edm\in\C{M}_R(U)$
satisfies (ii) as well and any weak solution of
$(H_0+\nu)\Phi=\lambda\cdot\Phi$ in $U$ is a weak solution of
$(H_0+\nu-Edm)\Phi=0$ in $U$. Thus it suffices to consider the case $\lambda=0$.
\item[(4)] By regularity we can replace (iii) by
$\aaa[\Phi,\varphi]+\int_U\varphi \tilde{\Phi}d\nu=\lambda \cdot (\Phi |\varphi)$
for all $\varphi \in \DD_{\text{loc}} \cap L^\infty_c (U)$ (see \cite{LenzSV} for details).
\end{itemize}
\end{remark}

\section{Positive weak solutions and the infimum of the spectrum} \label{Positive}
Throughout this section we consider a strongly local, regular Dirichlet form,
$(\aaa,\DD)$ on $X$ and denote by $\Gamma: \DD_{\text{loc}}\times\DD_{\text{loc}}\to
\C{M}(X)$ the associated energy measure. The results  discussed in this section are taken from \cite{LenzSV}  to which we refer for further details and proofs.

\subsection*{Ground state transform and consequences}

We start with a theorem giving the so called ground state transform in our general setting.

\begin{theorem}\label{t:transformation_global}
Let $(\aaa,\DD)$ be a regular, strictly local Dirichlet form, $H_0$ be the associated operator and
$\nu$ a measure with $\nu_+ \in \mathcal{M}_{R,0}, \nu_-\in \mathcal{M}_{R,1}$.  Suppose
that $\Phi$  is a weak solution  of $(H_0+\nu)\Phi=\lambda\cdot\Phi$ in $X$ with $\Phi>0$ $m$-a.e. and $\Phi, \Phi^{-1}\in L^\infty_{\text{loc}}(X)$.
Then, for   all $\varphi,\psi \in D(\cE + \nu)$, the products
 $\varphi \phim,\psi \phim$ belong to $\DD_{\text{loc}}$ and
 the formula
\begin{equation}\label{eq:formula}
\aaa[\varphi,\psi]+\nu[\varphi,\psi] = \int_X\Phi^2d\Gamma(\varphi\Phi^{-1},
\psi\Phi^{-1}) + \lambda\cdot (\varphi |\psi)
\end{equation}
holds.
\end{theorem}

The proof of the theorem proceeds essentially in two steps. In the first step a local version of the theorem is proven for 'smooth' $u,v$. In the second step this local version is then extended to the whole space. Note also that the conditions on $\Phi$ in the theorem imply that $\phim$ is in $\DD_{\text{loc}}$. As the local version may be of independent interest and has a very simple proof we include statement and proof next.

\begin{theorem}\label{t:transformation_local}
Let $(\aaa,\DD)$ be a regular, strictly local Dirichlet form, $H_0$ be the associated operator and $\nu\in\C{M}_{R,0}(U)$. Suppose
that $\Phi$  is a weak solution  of $(H_0+\nu)\Phi=\lambda\cdot\Phi$ in $U$ with $\Phi>0$ $m$-a.e. and $\Phi, \Phi^{-1}\in L^\infty_{\text{loc}}(U)$.
Then, for all $\varphi, \psi\in \DD\cap L^\infty_{c}(U)$:
\[
\aaa[\varphi,\psi]+\nu[\varphi,\psi] = \int_U\Phi^2d\Gamma(\varphi\Phi^{-1},
\psi\Phi^{-1}) + \lambda\cdot (\varphi |\psi) .
\]
\end{theorem}
\begin{proof} For the proof we may assume $\lambda=0$ without restriction.  Without loss of generality we may also assume that $\varphi$ and $\psi$ are real valued functions.  We now  evaluate the RHS of the above equation, using the following
identity. The Leibniz rule implies that  for arbitrary $w\in \DD_{\text{loc}}(U)$:
$$
0=d\Gamma(w,1)=d\Gamma(w,\Phi\Phi^{-1})=\Phi^{-1}d\Gamma(w,\Phi)+\Phi d\Gamma(w,\Phi^{-1})\qquad( \bigstar)
$$
Therefore, for $\varphi, \psi\in \DD\cap C_c(X)$:
\begin{eqnarray*}
\int_X\Phi^2d\Gamma(\varphi\Phi^{-1},
\psi\Phi^{-1}) &=& \int_X\Phi d\Gamma(\varphi,\psi\Phi^{-1})+
\int_X\Phi^2\varphi d\Gamma(\Phi^{-1},\psi\Phi^{-1})\\
(\mbox{by symmetry})\;\:&=& \int_X d\Gamma(\varphi,\psi)+\int_X\Phi \psi d\Gamma(\varphi,\Phi^{-1})\\
& & + \int_X\Phi^2\varphi d\Gamma(\psi\Phi^{-1},\Phi^{-1})\\
&=& \aaa[\varphi,\psi] + \int_X\Phi^2 d\Gamma(\varphi\psi\Phi^{-1},\Phi^{-1})\\
(\mbox{ by $(\bigstar)$})\;\:&=& \aaa[\varphi,\psi]- \int_X d\Gamma(\varphi\psi\Phi^{-1},\Phi)\\
 &=& \aaa[\varphi,\psi]-\aaa[\varphi\psi\Phi^{-1},\Phi].
 \end{eqnarray*}
 As $\Phi$ is a weak solution we can now use part (4) of the previous remark to continue the computation by
 \begin{eqnarray*}
...&=& \aaa[\varphi,\psi]-\left( -\nu[\varphi\psi\Phi^{-1},\Phi] \right)\\
 &=& \aaa[\varphi,\psi]+\nu[\varphi,\psi] .
\end{eqnarray*}
This finishes the proof.
\end{proof}

The ideas behind our proof allow for some further generalizations. This is shortly indicated in the following remark.
textbf{Remark.}
\begin{itemize}
\item[(1)] This proof actually shows the following statement: Assume that there is a weak supersolution $\Phi$ of $(H_0+\nu)\Phi=\lambda\cdot\Phi$ on $X$ with
$\Phi>0$ $m$-a.e. and $\Phi,\Phi^{-1}\in L^\infty_{\text{loc}} (X)$. Then
$\aaa +\nu \ge \lambda$.
\item[(2)] We can allow for complex measures $\nu$ without problems. In the context of PT--symmetric
  operators there is recent interest in this type of Schr\"{o}dinger
  operators, see \cite{BenderBM-99}
\item[(3)] Instead of measures also certain distributions could be included. Cf \cite{HerbstS-78} for such singular perturbations.
\end{itemize}

We explicitely note the following immediate consequence of (both) of the theorems of this section.

\begin{corollary}\label{AP-leichteRichtung} Let $(\aaa,\DD)$ be a regular, strictly local Dirichlet form, $H_0$ be the associated operator and $\nu$ a measure on $X$ with $\nu_+ \in \mathcal{M}_{R,0}, \nu_-\in \mathcal{M}_{R,1}$.  Suppose
that $\Phi$  is a weak solution  of $(H_0+\nu)\Phi=\lambda\cdot\Phi$ in $X$ with $\Phi>0$ $m$-a.e. and $\Phi, \Phi^{-1}\in L^\infty_{\text{loc}}(X)$. Then, $H_0 + \nu\geq \lambda$.
\end{corollary}

\subsection*{Harnack principles and existence of positive solutions below the spectrum}
The previous subsection shows that  $H_0+\nu\ge \lambda$ whenever
$\aaa+\nu$ is closable and admits a positive weak solution of
$(H_0+\nu)\Phi=\lambda\Phi$. In this subsection discuss the  converse under suitable
conditions.
A key property is related to the celebrated \emph{Harnack inequality}.
\begin{definition}\label{d:Harnack}
{\rm
\begin{itemize}
\item [(1)] We say that $H_0+\nu$ \emph{satisfies a Harnack inequality} for
  $\lambda\in\RR$ if, for every relatively compact, connected open $X_0\subset X$ there is a
  constant $C$ such that all positive weak solutions $\Phi$ of
  $(H_0+\nu)\Phi=\lambda\Phi$ on $X_0$ are locally bounded and satisfy
$$
\mbox{esssup}_{B(x,r)}u\le C  \mbox{essinf}_{B(x,r)}u ,
$$
for every $B(x,r)\subset X_0$ where esssup and essinf denote the essential supremum and infimum.
\item [(2)] We say that $H_0+\nu$ satisfies the \emph{Harnack principle} for
  $\lambda \in\RR$ if for every relatively compact, connected open subset $U$ of $X$
  and every sequence $(\Phi_n)_{n\in\NN}$ of nonnegative solutions of
  $(H_0+\nu)\Phi=\lambda\cdot\Phi$ in $U$ the following implication holds: If, for
  some measurable subset $A \subset U$ of positive measure
$$
\sup_{n\in\NN}\|\Phi_n{\mathds{1}}_A\|_2 <\infty
$$
then, for all compact $K\subset U$ also
$$
\sup_{n\in\NN}\| \Phi_n{\mathds{1}}_K\|_2 <\infty .
$$
\item [(3)]
 We say that $H_0+\nu$ satisfies the \emph{uniform Harnack principle}
 if for every bounded intervall $I\subset\RR$,
 every relatively compact, connected open subset $U$ of $X$ and every sequence $(\Phi_n)_{n\in\NN}$ of nonnegative solutions of $(H_0+\nu)\Phi=\lambda_n\cdot\Phi$ in $U$ with $\lambda_n\in I$ the following implication holds:
If, for some measurable subset $A \subset U$ of positive measure
$$
\sup_{n\in\NN}\|\Phi_n{\mathds{1}}_A\|_2 <\infty
$$
then, for all compact $K\subset U$ also
$$
\sup_{n\in\NN}\| \Phi_n{\mathds{1}}_K\|_2 <\infty .
$$
\end{itemize}
}
\end{definition}
Note that validity of a Harnack principle implies that a nonnegative
weak solution $\varPhi$ must vanish identically if it vanishes on a
set of positive measure (as $\varPhi_n := n \varPhi$ has vanishing
$L^2$ norm on the set of positive measure in question).  Note also
that validity of an Harnack inequality extends from balls to compact
sets by a standard chain of balls argument.  This easily shows that
$H_0+\nu$ satisfies the Harnack principle for $\lambda\in\RR$ if it obeys a
Harnack inequality for $\lambda\in\RR$.  Therefore, many situations are
known in which the Harnack principle is satisfied:

For $\nu\equiv 0$ and $\lambda=0$ a Harnack inequality holds, whenever $\aaa$ satisfies a Poincar\'e and a volume doubling property; cf \cite{BiroliM-95} and the discussion there. The most general results for $H_0=-\Delta$ in terms of the measures
  $\nu$ that are allowed seem to be found in \cite{Hansen-99a}. The uniformity
 of the estimates from \cite{Hansen-99a} immediately gives that the
  uniform Harnack principle is satisfied for Kato class measures.
   Of the enormous
  list of papers on Harnack's inequality, let us  also mention
  \cite{AizenmanS-82,Biroli-01,BiroliM-06, ChiarenzaFG-86,Hansen-99a, HoffmannHN-95, Kassmann-07,
    Moser-61,Saloff-Coste-95, Serrin-64,Sturm-94c,Sturm-96}

Apart from the Harnack principle there is a second property that will be
important in the proof of existence of positive general eigensolutions at
energies below the spectrum.

\begin{definition}
The form  $\aaa$ satisfies the \emph{local compactness property} if
$D_0(U):=\overline{D\cap C_c(U)}^{\|\cdot\|_\aaa}$ is compactly
embedded in $L^2(X)$ for every relatively compact open $U\subset X$.
\end{definition}
In case of the classical Dirichlet form the local compactness property  follows from Rellich's
Theorem on compactness of the embedding of Sobolev spaces in $L^2$.

\smallskip

It turns out that the situation is somewhat different depending on whether $X$ is compact or not.
 In both cases we will need the assumption of irreducibility in order to obtain solutions which are positive almost everywhere. This is clear as in the reducible case a nontrivial  solution could still  vanish on some 'components'.

  We first get the case of compact $X$ out of our way.

\begin{theorem}\label{AP-schwereRichtung-compact}
  Let $(\aaa,\DD)$ be a regular, strictly local, irreducible Dirichlet
  form, $H_0$ be the associated operator and
  $\nu$ a measure with $\nu_+ \in \mathcal{M}_{R,0}, \nu_-\in \mathcal{M}_{R,1}$.  Suppose that $X$ is compact and  $\aaa$ satisfies the
  local compactness property. Then, $H_0 + \nu$ has compact resolvent. In particular,   there exists a positive weak solution
  to $(H_0 + \nu) \Phi = \lambda_0 \Phi$ for $\lambda_0:=\inf \sigma (H_0 +\nu)$.  This solution is unique (up to a factor) and  belongs to $L^2
  (X)$. If $H_0 + \nu$ satisfies a Harnack principle, then  $\lambda_0$ is the only value in $\RR$ allowing for a positive  weak solution.
\end{theorem}

We can now state our result in the case of non-compact $X$.

\begin{theorem}\label{AP-schwereRichtung-non-compact}
  Let $(\aaa,\DD)$ be a regular, strictly local, irreducible Dirichlet
  form, $H_0$ be the associated operator and
  $\nu$ with $\nu_+ \in \mathcal{M}_{R,0}, \nu_-\in \mathcal{M}_{R,1}$.  Suppose that $\aaa$ satisfies the
  local compactness property and $X$ is noncompact. Then, if
  $\lambda<\inf\sigma(H_0+\nu)$ and $H_0+\nu$ satisfies the Harnack
  principle for $\lambda $, there is an a.e. positive solution of
  $(H_0+\nu)\Phi=\lambda \Phi$.
\end{theorem}

  That we have to assume that $X$ is noncompact can easily be seen by
  looking at the Laplacian on a compact manifold. In that situation
  any positive weak solution must in fact be in $L^2$ due to the
  Harnack principle. Thus the corresponding energy must lie in the
  spectrum (see Theorem \ref{AP-schwereRichtung-compact}).

\subsection*{Characterizing the infimum of the spectrum}
The previous results to not yet settle the existence of a positive weak solution for the groundstate energy $\inf \sigma (H_0+\nu)$ in the noncompact case. The uniform Harnack principle settles this question:

\begin{theorem}\label{Characterization-AP} Let $(\aaa,\DD)$ be a regular, strictly local, irreducible Dirichlet form, $H_0$ be the associated operator,
$\nu$ with $\nu_+ \in \mathcal{M}_{R,0}, \nu_-\in \mathcal{M}_{R,1}$.
  Suppose that $\aaa$ satisfies the local compactness property and $H_0+\nu$ satisfies the uniform Harnack principle. Then there is an a.e. positive weak
solution of $(H_0+\nu)\Phi=\lambda \Phi$ for  $\lambda = \inf \sigma (H_0+\nu)$.
\end{theorem}

\section{Weak solutions and spectrum}\label{Weaksolutions}
In this section we relate energies in the spectrum to energies for which   (suitably bounded) weak solutions exist. The results are taken from  \cite{BoutetdeMonvelLS-08}. The final characterization relies on  \cite{BoutetdeMonvelS-03b} as well.

\subsection*{A Weyl type criterion}\label{Weyl}
We include the following criterion for completeness. It is taken from
\cite{Stollmann-01}, Lemma 1.4.4, see also  \cite{DermenjianDI-98}, Lemma 4.1 for the same result in a slightly different formulation.

\begin{proposition}\label{PropWeyl} Let $h$ be a closed, semibounded form and
  $H$ the associated selfadjoint operator. Then the following
  assertions are equivalent:

  \begin{itemize}

  \item[(i)] $\lambda \in \sigma(H)$.

  \item[(ii)] There exists a sequence $(u_n)$ in $\mathcal{D}(h)$ with
    $\|u_n\|\to1$ and
$$ \sup_{v\in \mathcal{D}(h),\|v\|_h \leq 1} |(h -\lambda) [u_n,v ]| \to 0,$$
for $n\to \infty$.

\end{itemize}

\end{proposition}

\subsection*{A Caccioppoli type inequality}
In this section we prove a bound on the energy measure of a
generalized eigenfunction on a set in terms of bounds on the
eigenfunction on certain neighborhood of the set.

We need the following notation: For $E\subset X$ and $b>0$ we define the
$b$-neighborhood of $E$ as
$$ B_b (E) :=\{y\in X : \rho(y,E)\leq b\}.$$

\begin{theorem}\label{Caccioppoli}
  Let $\mathcal{E}$ be a strictly  local regular Dirichlet form.  Let $\mu_+\in \mathcal{M}_0$ and
  $\mu_- \in \mathcal{M}_1$ be given. Let $\lambda_0\in \RR$ and $b_0>0$ be given.
  Then, there exists a $C = C (b_0,\lambda_0,\mu_-)$ such that for any
  generalized eigenfunction $u$ to an eigenvalue $\lambda\leq
  \lambda_0$ of $H_0 + \mu$ the inequality
$$ \int_E d\Gamma (u) \leq \frac{C}{b^2} \int_{B_b (E)} |u|^2
dm$$ holds for any closed $E\subset X$ and any $0<b\le b_0$.
\end{theorem}

\textbf{Remark.} For compact $E$ both sides in the above inequality are finite, for $E$ merely closed, one or both sides might be infinite. In any case, it suffices to prove the compact case since $\Gamma$ is a Radon measure.

The Caccioppoli inequality replaces the familiar
commutator estimates that are used for Schr\"odinger operators.

\subsection*{A $\frac12$ Shnol type result: How suitably bounded solutions force spectrum}

In this section, we first present an abstract Shnol type result. Unfortunately, we have to start with a disclaimer. In \cite{BoutetdeMonvelLS-08} we messed up the reference to Shnol's original result (as do many other authors). In fact, \cite{Shnol-54} is the correct citation but there are two more papers with quite similar titles  \cite{Shnol-54b,Shnol-57} and \cite{Shnol-54}  does not appear in MathSciNet.

The latter article deals with Schr\"odinger operators on the half line and says that for spectrally almost every $\lambda\in\RR$ the solution on the eigenvalue problem is bounded by $const\,  x^{\frac12 +\varepsilon}$ as $x\to\infty$ and vice versa. By ``the solution'' we mean a solution with the prescribed boundary condition at $0$ and such a solution always exists since we are dealing with ODE.
In this section we show $\frac12$ Shnol, even a little stronger: if a weak solution with suitable exponential bounds exist for a given energy, that energy is in the spectrum.

We
need the following notation.  For $E\in X$ and $b>0$ we define the inner
$b$-collar of $E$ as
$$ C_b (E) :=\{y\in E :  \rho(y,E^c)\leq b\}.$$

\begin{theorem}\label{SchnolAbstract} Let $\mathcal{E}$ be a strictly local regular Dirichlet form.  Let $\mu_+\in \mathcal{M}_0$ and
  $\mu_- \in \mathcal{M}_1$ be given. Let $\lambda\in \RR$ with
  generalized eigenfunction $u$ be given. If there exists $b>0$ and a
  sequence $(E_n)$ of compact subsets of $X$ with
$$\frac{\| u \chi_{C_{b}(E_n)} \|}{\| u \chi_{E_n}\|}\longrightarrow
0,n\longrightarrow 0,$$ then $\lambda$ belongs to $\sigma (H)$.
\end{theorem}

We will now specialize our considerations to subexponentially bounded
eigenfunctions.

\medskip

A function $J: [0,\infty) \longrightarrow [0,\infty)$ is said to
be subexponentially bounded if for any $\alpha>0$ there exists a
$C_\alpha\geq 0$ with $J(r) \leq C_\alpha \exp(\alpha r)$ for all
$r>0$.
A $\KK$-valued function $f$ on a pseudo metric space   $(X,\rho)$
with measure $m$ is said to be subexponentially bounded if for some $x_0\in X$ and
$\omega (x) = \rho (x_0,x)$ the function $e^{-\alpha \omega} u$ belongs to  $ L^2 (X,m)$ for any $\alpha >0$.
Recall that a strictly local regular Dirichlet form $\cE$ gives rise to an intrinsic pseudo metric $\rho$ and
an associated pseudo metric space $(X,\rho)$.

\begin{theorem}\label{SchnolConcrete} Let $\mathcal{E}$ be a strictly  local regular Dirichlet form, $x_0\in X$
  arbitrary and $\omega (x) = \rho (x_0,x)$. Let $\mu_+\in \mathcal{M}_0$ and
  $\mu_- \in \mathcal{M}_1$ be given. Let $u\not=0$ be a subexponentially bounded  generalized
  eigenfunction.
Then, $\lambda$
  belongs to $\sigma(H)$.
\end{theorem}

\subsection*{A $\frac12$ Shnol type result: How spectrum forces suitably bounded generalized eigenfunctions.}
In the last subsection we have discussed that existence of suitably bounded weak solutions implies that an energy belongs to the spectrum. In this section we discuss a converse given in \cite{BoutetdeMonvelS-03b} that was known before for ordinary  Schr\"odinger operators; see the literature cited in the monograph \cite{Berezanskii-68}.

Recall that $\aaa$ is called \textit{ultracontractive} if  for each $t>0$ the  semigroup  $e^{-tH_0}$ gives a map from  $L^2(X)$ to $L^\infty (X)$.

\begin{theorem}\label{T46}
  Let $\mathcal{E}$ be a strictly  local, regular, ultracontractive  Dirichlet   form
  satisfying condition  {\rm {(G)}}. Let $\mu=\mu_+-\mu_-$
  with $\mu_+\in\C{M}_0$ and $\mu_-\in \hat{\mathcal{S}}_K$ with $
  c_{Kato}(\mu)<1$. Define $H:=H_0+\mu$. Then for spectrally a.e.
  $\lambda\in \sigma(H)$ there is a subexponentially bounded
  generalized eigenfunction $u\not= 0$ with $Hu=\lambda u$.
\end{theorem}

Actually, as remarked in \cite{BoutetdeMonvelS-03b}, one does arrive at generalized eigenfunctions with polynomial bounds if one assumes that the volume of balls grows polynomially as well.

\subsection*{A Shnol type result: Characterizing the spectrum by subexponentially bounded solutions}
We can now put together the results of the preceeding subsections and obtain a characterization of the spectrum via subexponentially bounded solutions.

\medskip

\begin{corollary}\label{Characterization-Schnol}
  Let $\mathcal{E}$ be a stricly  local, regular, ultracontractive Dirichlet form
  satisfying {\rm {(G)}}. Let $\mu=\mu_+-\mu_-$
  with $\mu_+\in\C{M}_0$ and $\mu_-\in \hat{\mathcal{S}}_K$ with $
  c_{Kato}(\mu)<1$. Define $H:=H_0+\mu$. Then the spectral measures of
  $H$ are supported on
$$
\{ \lambda\in\RR | \exists \mbox{ subexponentially bounded }u\not=0\mbox{
  with }Hu=\lambda u\} .
$$
\end{corollary}

\section{Examples and applications}\label{Examples}
Several different types of operators to which our results can be
applied have already been mentioned in Section 1. This includes
classical examples like Schr\"odinger operators and symmetric
elliptic second order differential operators on unbounded domains in
$\RR^d$.
More generally, Laplace-Beltrami operators and rather general
elliptic second order differential operators on Riemannian manifolds
fall also within this class. In this section we will discuss in some more
detail two types of examples which have attracted attention more
recently, namely singular interaction Hamiltonians and quantum graphs.
Moreover, we discuss here applications of the ground state transformation.

\subsection*{Hamiltonians with singular interactions}

Hamiltonians with singular interactions arise when the Laplacian is perturbed
by a perturbation which is localized on a set of Lebesgue measure zero.
Here we consider more specifically operators with an interaction supported on an
orientable, compact sub-manifold $M \subset \RR^d$ of class
$C^2$ and codimension one. The manifold $M $ may or may
not have a boundary.
In the sequel we follow roughly the exposition
in \cite{KondejV-06a}. For more background see
\cite{BrascheEKS-94} or Appendix K of \cite{AlbeverioGHH-04}.

\medskip

The simplest type of Hamiltonian with a potential perturbation supported on $M $ is
formally given by
\begin{equation} \label{e-delta}
(\Hag)f(x) := \Big(-\Delta  - \alpha \cdot \delta(x -M)\Big)f(x)\,,
\end{equation}
where $\alpha>0 $ is a coupling constant.
To show that the operator $\Hag$ can be given a rigorous
meaning we establish next that it falls into the framework outlined in Section \ref{Measure}.

For this purpose denote by $\nu_M $ the Dirac
measure in $\RR^d$ with support on $M$. This means that for
any Borel set $G\subset \RR^d$ we have $\nu_M (G)= s_{d-1}(G\cap M )$.
Here $s_{d-1}$ is the $d-1$ dimensional surface measure on
$M $.
From Theorem 4.1 in
\cite{BrascheEKS-94} we infer that the measure $\nu_M $
belongs to the Kato class. In particular, for such a measure
and an arbitrary $a>0$ there exists $b_a<\infty$ such that
\[
\int_{\RR^d} |\psi(x) |^2 \nu_M (dx)
\leq
a\|\nabla \psi \|^2 +b_a\|\psi\|^2\,.
\]
As mentioned in Section \ref{Measure} this implies that
the form  $\cE_{\alpha\nu_M}:= \cE + \alpha \nu_M$ is
closed on the domain $\cD$ and  densely defined.
The unique selfadjoint operator associated to
$\cE_{\alpha\nu_M}$ acing  on $L^2(\RR^d)$
will be denoted by $H_{\alpha\nu_M}$.

It is possible to define the operator $H_{\alpha\nu_M}$
by appropriate selfadjoint boundary conditions on
$M$, cf.{} \cite{BrascheEKS-94,KondejV-06a}.
To explain this more precisely we need some notation.
Denote by $\mathrm{n}\colon M \to \Sphere^d$ a global
unit normal vectorfield on $M$.
Denote by $D(\tilde{H}_{\alpha
\nu_M})$ the set of functions
\[
\psi\in C(\RR^d)\cap W^{1,2}(\RR^d) \cap
C^\infty(\RR^d\setminus M) \cap W^{2,2}(\RR^d\setminus M)
\]
which satisfy for all $ x \in M $
\begin{align*}
 \lim_{\epsilon \searrow 0} \frac{\psi(x+\epsilon
\mathrm{n}(x))-\psi(x)}{\epsilon}
+
 \lim_{\epsilon \searrow 0} \frac{\psi(x-\epsilon
\mathrm{n}(x))-\psi(x)}{\epsilon}
= - \alpha \, \psi(x)
\end{align*}
Using Green's formula one concludes as in Remark 4.1 of
\cite{BrascheEKS-94} that the closure of $-\Delta$ with domain
$D(\tilde{H}_{\alpha \nu_M})$ is the selfadjoint operator
$H_{\alpha \nu_M }$.

\bigskip

Since the measure $\nu_M$ belongs to the Kato class and is supported on a compact
set, the essential spectrum of $H_{\alpha\nu_M}$ equals $[0,\infty)$,
cf.~Theorem 3.2 in \cite{BrascheEKS-94}.
In space dimension two $H_{\alpha\nu_M}$ has nonempty discrete spectrum for any
positive value of the coupling constant $\alpha $.
This can be seen using the proof of Corollary 11 in \cite{Brasche-03}.
For higher dimensions there is a critical value $\alpha_c > 0$
such that there exists a negative eigenvalue
if and only if $\alpha \geq  \alpha_c$, cf.{} the discussion on page 20 of \cite{ExnerY-03}.

%
%
%

\subsection*{Quantum graphs}
Quantum graphs are given in terms of a metric graph $\G$ and a
Laplace (or more generally) Schr\"odinger operator $H$ defined on
the edges of $\G$ together with a set of (generalised) boundary
conditions at the vertices which make $H$ selfadjoint.
To make this more precise
we define the geometric structure of metric graphs,
as well as the operators acting on the associated $L^2$-Hilbert space.

We start with the definition of a metric graph which is appropriate for our purposes.
\begin{definition}
Let $V$ and $E$ be countable sets, $l_-$ a positive real, and $\cG$ a map
\[
\cG \colon E \to V\times V \times [l_-,\infty]
, \quad e \mapsto (\iota(e),\tau(e),l_e).
\]
Here $[l_-,\infty]$ means $[l_-,\infty)\cup \{+ \infty\}$.
We call the triple $\G=(V,E,\cG)$ a metric graph, elements of $V=V(\G)$ vertices, elements of $E=E(\G)$ edges,
$\iota(e)$ the initial vertex of $e$, $\tau(e)$ the terminal vertex of $e$ and
$l_e$ the length of $e$. Both $\iota(e)$ and $\tau(e)$ are called endvertices of $e$,
or incident to $e$.
The number of edges incident to the vertex $v$ is called the
\textit{degree of $v$}. We assume that the degree is finite for all vertices.
\end{definition}

Note that the two endvertices of an edge are allowed to coincide and there may be multiple edges
connecting two vertices.
We let $X_e:= \{ e\}\times (0, l_e)$, $X=X_\Gamma=V\cup\bigcup_{e\in E}X_e$
and $\overline{X_e}:=\{ e\}\times [0, l_e]$.
On the set $X$ it is possible to define in a natural way the length of paths and, using
this notion, also a metric, cf. Section 1 in \cite{LenzSS}.

Now we introduce the relevant Hilbert spaces on which the Laplace, respectively, Schr\"odinger operators
will act. For $ k \in \{0,1,2\}$ we set
\[
W^{k,2} (E) := \bigoplus_{e\in E } W^{k,2}(0, l_e)
\]
and for $W^{0,2} (E)$ we use the usual notation $L^{2} (E)$.
Given $ k \in \{0,1,2\}$ and a function  $u \in W^{1,k}(E)$
we denote by $u_e$ the projection of $u$ to the space $W^{k,2}(0,l_e)$.
Thus we can identify each $u \in W^{1,k}(E)$  with a family $(u_e)_{e\in E}, u_e\in W^{0,2}(0, l_e)$.

Next we discuss pointwise properties of functions in $u \in W^{1,k}(0, l_e)$.
Recall that for any  $l>0$  any element $h$ of $W^{1,2}(0,l)$ has a continuous version; we will always pick this version and then  the boundary value
$h(0):=\lim_{x\to 0+} h(x)$ exists and satisfies
\begin{equation}\label{Sobolev}
|h(0)|^2 \leq \frac{2}{l} \|h\|^2_{L^2 (0,l)} + l \|h'\|^2_{L^2 (0,l)}
\end{equation}
by standard Sobolev type theorems. Consider now an edge $e$, the vertex
$v=\iota(e)\in V$ and $u\in W^{1,2} (0,l_e)$.  Then the limit $ u (v) :=\lim_{t\to 0} u
(t) $ exists, as well as $u (w) :=\lim_{t\to l_e} u (t) $ for $w=\tau(e)$ and
\eqref{Sobolev} holds (with the obvious modifications).  Similarly, for an
edge $e$ and the vertex $v= \iota(e) $ and the vertex $w=\tau(e)$ and $u\in W^{2,2}
(0,l_e)$ the limits $ u' (v) :=\lim_{x\to v, x\in e} u' (x) $ and  $ u' (w)
:=-\lim_{x\to w, x\in e} u' (x)$ exist.
Note that our sign convention is such that
the definition of the derivative is canonical, i. e. independent of
the choice of orientation of the edge.  For $f\in W^{1,2}(E)$ and each vertex
$v$ we gather the boundary values of $f_e (v)$ over all edges $e$ adjacent to
$v$ in a vector $f(v)$.  More precisely, denote by $E_v:= \{ e\in E| v\in\{
\iota(e),\tau(e)\}\}$ the set of vertices adjacent to $v$ and define $f(v):=
(f_e(v))_{e\in E_v}\in \CC^{E_v}$ and similarly, for $f\in W^{2,2}(E)$ we
further collect the boundary values of $f_e' (v)$ over all edges $e$ adjacent
to $v$ in a vector $f' (v)\in \CC^{E_v}$.  These boundary values of functions
will be used to define the boundary conditions of the Laplacian, respectively the domains of definition of
the forms we will be considering. Here we restrict ourselves to Kirchhoff boundary conditions and call a function $(u_e)_{e\in E}\in W^{1,2}(E)$ continuous, if, for any vertex $v$ and all edges adjacent to it $u_e(v)=u_{e'}(v)$.
Now set
\begin{align}
 D(s_0) &:=  W^{1,2}(E)\cap C(X) \\
 s_0 (f,g) &:=\sum_{e\in E} \int_0^{l(e)} f_e' (t)\overline{g}_e'(t) dt
\end{align}
Obviously, the form $s_0$ is bounded below, closed a Dirichlet form and strongly local.
Hence, there exists a unique associated self-adjoint operator which we denote by $H_P$.
It can be explicitly characterized by
\begin{align*}
D(H_K) &:= \{f\in W^{2,2} (E)\cap C(X)  : \sum_{e\in E_v} f_e (v) =0\;\:\mbox{for all $v\in V$ } \}\\
 (H_K f)_e:&= - f_e'' \;\:\mbox{for all $e\in E$}.
\end{align*}
It is possible to define quantum graphs with more general generalised boundary conditions at the vertices but not all reasonable choices will lead to Dirichlet forms; in \cite{KantKVW-09} a characterization of those boundary conditions for which the form is a Dirichlet form is given. However the setup is somewhat different from ours.

\subsection*{Applications}\label{Applications}

The ground state transformation which featured in Theorem
\ref{t:transformation_global} can be used to obtain a formula for
the lowest spectral gap. To be more precise let us assume that
$\mathcal{E}$, $\nu$ and $\Phi$ satisfy the conditions of Theorem
\ref{t:transformation_global}. Assume in addition that $\Phi$ is in
$\mathcal{D} (\cE+ \nu)$. Then $\Phi$ is an eigenfunction of $H$
corresponding to the eigenvalue $\lambda  = \min \sigma (H)$. We denote by
\[
 \lambda ':= \inf \{ \aaa[u,u] + \nu[u,u] \mid u \in \DD, \|u \|=1, u \perp \Phi \} \,
\]
the second lowest eigenvalue below the essential spectrum of $H$,
or, if it does not exist, the bottom of $\sigma_{ess}(H)$. Then we
obtain the following formula
\begin{equation}\label{e:gap-formula}
 \lambda ' -\lambda  = \inf_{ \{ u \in \DD (\cE+\nu) , \|u \|=1, u \perp \Phi \}} \int_X \Phi^2 d\Gamma(u \Phi^{-1},u \Phi^{-1}) \
\end{equation}
which determines the lowest spectral gap. It has been used in
\cite{KirschS-87,KondejV-06a,Vogt} to derive lower bounds on the
distance between the two lowest eigenvalues of different classes of
Schr\"odinger operators (see \cite{SingerWYY-85} for a related
approach). In \cite{KirschS-87} bounded potentials are considered,
in \cite{KondejV-06a} singular interactions along curves in $\RR^2$
are studied, and \cite{Vogt} generalizes these results using a
unified approach based on Kato-class measures.

If for a subset  $U \subset X$ of positive measure and a function
 $u \in \DD$ with $ \|u \|=1$ and $ u \perp \Phi \}$ the non-negative
measure $\Gamma(u\Phi^{-1},u\Phi^{-1})$ is absolutely continuous
with respect to $m$, one can exploit formula \eqref{e:gap-formula}
to derive the following estimate (cf.{} Section~3 in \cite{Vogt},
and \cite{KirschS-87,KondejV-06a} for similar bounds). Denote by
$\gamma(u\Phi^{-1})= \frac{d\Gamma(u\Phi^{-1},u\Phi^{-1})}{dm}$ the
Radon-Nykodim derivative. Then an application of the Cauchy-Schwarz
inequality gives
\[
\int_U \Phi^2 d\Gamma(u \Phi^{-1},u \Phi^{-1}) \ge \frac{1}{m(U)}
\inf_U \Phi^2
  \left( \int_U \sqrt{\gamma(u\Phi^{-1})}dm\right)^2 \,
\]

Now we formulate more precisely the setting
in which the above mentioned results \cite{KirschS-87,KondejV-06a,Vogt}
apply. In fact, we choose here to formulate the main theorem of
\cite{Vogt}. It applies to more general situations than \cite{KirschS-87} and \cite{KondejV-06a}
and is formulated in the language of Dirichlet forms.
Consider the case where $X= \RR^d$, $\cE$ is equal to the classical Dirichlet form,
 $\nu$ is a non-negative, compactly supported measure satisfying for some
$c_\nu \in (0,\infty), \alpha \in [0,2)$ the bound
$\nu(B(x,r)) \le c_\nu r^{d-\alpha}$ for all $r>0, x \in \RR^d$,
and $D$ denotes the diameter of the support of $\nu$.
Let us assume that the bottom of the spectrum of $\cE + \nu$ consists
of two isolated eigenvalues, which will be denoted by $\lambda_0 < \lambda_1$.
Under these assumptions there exist constants $C,C_0, p,q \in (0,\infty)$
such that
\[
\lambda_1-\lambda_0 \ge \frac{C}{(c_\nu+1)^p(D+1)^q} \cdot |\lambda_0| \cdot e^{-C_0 (D+1) \cdot \sqrt{|\lambda_0|}}
\]

\bigskip

The ground state transformation plays an important role in other
situations as well.
It is for instance used in the the study of $L^p$-$L^q$ mapping properties of the semigroup associated to $\aaa$
\cite{DaviesS-84}. In the theory of random Schr\"odinger operators it is used to
remove a symmetry condition from the proof of
of Lifschitz tails \cite{Mezincescu-87}.\\[5mm]

\def\cprime{$'$}\def\polhk#1{\setbox0=\hbox{#1}{\ooalign{\hidewidth
  \lower1.5ex\hbox{`}\hidewidth\crcr\unhbox0}}}

\end{document}